\documentclass[11pt,a4paper]{article}
\usepackage[left=1.8cm, top=2.45cm,bottom=2.45cm,right=1.8cm]{geometry}
\usepackage{mathtools,amssymb,amsthm,mathrsfs,calc,graphicx,xcolor,cleveref,dsfont,tikz,pgfplots,bm,url,tabularx}
\usepackage[british]{babel}
\usepackage{amsfonts}              
\usepackage[T1]{fontenc}

\numberwithin{equation}{section}
\numberwithin{figure}{section}
\numberwithin{table}{section}

\theoremstyle{plain}
\newtheorem{theorem}{Theorem}[section]

\newtheorem{corollary}[theorem]{Corollary}
\newtheorem{proposition}[theorem]{Proposition}

\theoremstyle{definition}

\theoremstyle{remark}
\newtheorem{remark}[theorem]{Remark}

\newcommand{\bd}{\operatorname{bd}}

\newcommand{\apex}{\mathop{\mathrm{apex}}\nolimits}
\newcommand{\ext}{\mathop{\mathrm{ext}}\nolimits}
\newcommand{\skel}{\mathop{\mathrm{skel}}\nolimits}
\newcommand{\vol}{\mathop{\mathrm{vol}}\nolimits}
\newcommand{\supp}{\mathop{\mathrm{supp}}\nolimits}

\newcommand{\bx}{\mathbf{x}}

\def\EE{\mathbb{E}}

\def\JJ{\mathbb{J}}

\def\NN{\mathbb{N}}

\def\PP{\mathbb{P}}

\def\RR{\mathbb{R}}

\def\TT{\mathbb{T}}

\def\cC{\mathcal{C}}

\def\cF{\mathcal{F}}
\def\cG{\mathcal{G}}

\def\cL{\mathcal{L}}

\def\cT{\mathcal{T}}

\def\cV{\mathcal{V}}
\def\cW{\mathcal{W}}

\newcommand{\dd}{{\rm d}}

\newcommand{\pow}{\mathop{\mathrm{pow}}\nolimits}

\newcommand{\ii}{{\rm{i}}}

\newcommand{\eps}{\varepsilon}

\makeatletter
\let\@fnsymbol\@alph
\makeatother

\begin{document}


\title{\bfseries Sectional Voronoi tessellations: \\ Characterization and high-dimensional limits}

\author{Anna Gusakova\footnotemark[1],\; Zakhar Kabluchko\footnotemark[2],\; and Christoph Th\"ale\footnotemark[3]}

\date{}
\renewcommand{\thefootnote}{\fnsymbol{footnote}}
\footnotetext[1]{M\"unster University, Germany. Email: gusakova@uni-muenster.de}

\footnotetext[2]{M\"unster University, Germany. Email: zakhar.kabluchko@uni-muenster.de}

\footnotetext[3]{Ruhr University Bochum, Germany. Email: christoph.thaele@rub.de}

\maketitle

\begin{abstract}
\noindent
The intersections of beta-Voronoi, beta-prime-Voronoi and Gaussian-Voronoi tessellations in $\mathbb{R}^d$ with $\ell$-dimensional affine subspaces, $1\leq \ell\leq d-1$, are shown to be random tessellations of the same type but with different model parameters.
In particular, the intersection of a classical Poisson-Voronoi tessellation with an affine subspace is shown to have the same distribution as a certain beta-Voronoi tessellation. The geometric properties of the typical cell and, more generally, typical $k$-faces, of the sectional Poisson-Voronoi tessellation are studied in detail. It is proved that in high dimensions, that is as $d\to\infty$, the intersection of the $d$-dimensional Poison-Voronoi tessellation with an affine subspace of fixed dimension $\ell$ converges to the $\ell$-dimensional Gaussian-Voronoi tessellation. \\

\end{abstract}

\section{Introduction}

The present paper is devoted to the study of affine sections of Poisson-Voronoi tessellations. To define these,
let $\zeta\subset\RR^d$ be the set of atoms of a stationary point process in $\RR^d$.
For each point $x\in\zeta$ we construct the \textbf{Voronoi cell} $V(x,\zeta)$ of $x$ as the set of all points in $\RR^d$ which are closer to $x$ than to any other point of $\zeta$:
$$
V(x,\zeta) := \{z\in\RR^d:\|x-z\|\leq\|y-z\|\quad\text{for all }y\in\zeta\setminus\{x\}\},
$$
where $\|\,\cdot\,\|$ stands for the Euclidean {norm. The Voronoi cell} can be thought of as a zone of influence or attraction of the point $x$ and it is known that each Voronoi cell is a convex polytope in $\RR^d$, with probability $1$. The collection of all these polytopes is the \textbf{Voronoi tessellation} associated with $\zeta$. If $\zeta$ is a homogeneous Poisson point process with constant intensity $\rho>0$, this construction yields the   \textbf{Poisson-Voronoi tessellation}, denoted here by $\cW_{d,\rho}$,  -- one of the most classical models studied in stochastic geometry. We refer to the monographs \cite{OkabeEtAl,SW,SKM} for more detailed information, applications and further references on  Voronoi tessellations and in particular the Poisson-Voronoi tessellation.

A number of random tessellations studied in stochastic geometry, such as the Poisson hyperplane or the STIT tessellations, have the distinguished feature of being stable under intersections with lower-dimensional affine subspaces. By this we mean that the intersection with an affine subspace of one of these random tessellations is again a model of the same type within the intersecting subspace. For example, the intersection of a Poisson hyperplane tessellation with an affine subspace $L$ is again a Poisson hyperplane tessellation within $L$. However, a similar property is not true for the Poisson-Voronoi tessellation. In fact, it has been shown by Chiu, Van De Weygaert and Stoyan \cite{ChiuEtAlSection} that the intersection of the Poisson-Voronoi tessellation with an affine subspace cannot be a Voronoi tessellation induced by any stationary point process within the subspace. In other words, the sectional Poisson-Voronoi tessellation is necessarily a `non-Voronoi' tessellation. However, besides a few mean values determined in \cite{ChiuEtAlSection,MilesSection} further probabilistic or geometric information about the sectional Poisson-Voronoi tessellation seems not available in the existing literature, although they are of importance for stereological applications (see \cite[Chapter 11.5.4]{SKM} or \cite[Section 14.4.6]{NagelSurvey} as well as the references cited therein). It is one of the main purposes of this paper to derive a precise description of the sectional Poisson-Voronoi tessellation and to study its typical cell. We do this by establishing a connection with the so-called $\beta$-Voronoi tessellations, a random tessellation model we recently introduced and studied in the series of papers \cite{GKT21b,GKT20,GKT21,GKT21a}. Their analysis in turn was based on the connection with the class of beta random polytopes, which has already seen a number of applications in stochastic geometry \cite{kabluchko_formula,KTT,kabluchko_thaele_voronoi_sphere,beta_polytopes}.

We study the problem of the sectional Poisson-Voronoi tessellation just explained in a more general framework. In fact, the random tessellation we study is either
\begin{itemize}
\item a $\beta$-Voronoi tessellation $\cV_{d, \beta, \gamma}$ in $\RR^d$ with parameters $\beta\geq -1$ and $\gamma>0$,
\item a $\beta'$-Voronoi tessellation $\cV_{d, \beta, \gamma}'$ in $\RR^d$ with parameters $\beta > {d\over 2}+1$ and $\gamma>0$,
\item or a Gaussian-Voronoi tessellation $\cG_{d,\lambda}$ in $\RR^d$ with parameter $\lambda>0$;
\end{itemize}
a description of all these models will be provided in Section \ref{sec:betadefinition}. We remark that the classical Poisson-Voronoi tessellation generated by a stationary Poisson point process in $\RR^d$ with intensity $\rho>0$ appears in this framework as the $\beta$-Voronoi tessellation corresponding to the parameters $\beta=-1$ and {$\gamma=\pi^{{d+1\over 2}} \rho / \Gamma({d+1\over 2})$.}
Now, let $L\subset\RR^d$ be an affine subspace of dimension $1\leq\ell\leq d-1$. We show in Theorem \ref{tm:betaintersection} below that
\begin{itemize}
\item the sectional tessellation $\cV_{d, \beta, \gamma}\cap L$ is a $(\beta+{d-\ell\over 2})$-Voronoi tessellation in $L$ with the same $\gamma$,
\item the sectional tessellation $\cV_{d, \beta, \gamma}'\cap L$ is a $(\beta-{d-\ell\over 2})'$-Voronoi tessellation in $L$ with the same $\gamma$,
\item the sectional tessellation $\cG_{d,\lambda} \cap L$ is again a Gaussian-Voronoi tessellation in $L$ with the same $\lambda$.
\end{itemize}
In particular, the intersection of the classical Poisson-Voronoi tessellation $\cW_{d,\rho}$ with $L$ turns out to be a $\beta$-Voronoi tessellation within $L$ with $\beta={d-\ell\over 2}-1$ and {$\gamma = \pi^{{d+1\over 2}} \rho /\Gamma({d+1\over 2})$.} For clarity we should remark that none of the random tessellations $\cV_{d,\beta, \gamma}$ with $\beta>-1$ are actually  Voronoi tessellations. Our terminology is motivated by the fact that $\cV_{d,\beta,\gamma}$ may be viewed as a \textit{deformation} of the classical Poisson-Voronoi tessellation (which corresponds to $\beta=-1$).

With the identification of the sectional Poisson-Voronoi tessellation at hand, in the second part of this paper we study its geometric properties. More precisely, we determine in Theorem \ref{tm:cellintencity} its face intensities in terms of quantities which have already appeared in the study of beta random polytopes \cite{kabluchko_formula,beta_polytopes}. From here on, we determine the expected volume, the expected intrinsic volumes as well as the expected $f$-vector of the typical cell (and even more generally the typical $k$-face) of the sectional Poisson-Voronoi tessellation. Moreover, generalizing earlier results of Miles \cite{MilesSection} we consider the asymptotics, as $d\to\infty$,  of several characteristics (such as the volume of the typical cell) of the sectional Poisson-Voronoi tessellation and identify the limits with the corresponding characteristics of a suitable Gaussian-Voronoi tessellation. The weak convergence on the level of tessellations is discussed as well using a coupling construction similar to the one in \cite{GKT21}.

\section{Preliminaries on random tessellations}\label{sec:preliminaries_tess}
In this section we collect some definitions and facts about general stationary random tessellations in $\RR^d$. For more detailed discussions we refer the reader to \cite[Chapters 4 and 10]{SW} as well as \cite[Chapter 10]{SKM}.

A \textbf{tessellation} $T$ in $\RR^{d}$ is a countable, locally finite collection of $d$-dimensional polytopes, which cover the space and have non-empty, disjoint interiors. The elements of $T$ are called the {\bf cells} of $T$. Given a polytope $c\subset\RR^d$ we denote by $\cF_k(c)$ the set of its $k$-dimensional faces, $0\leq k\leq d$, where $\cF_d(c)=\{c\}$, and let $\cF(c):=\bigcup_{k=0}^{d}\cF_{k}(c)$. A tessellation $T$ is called {\bf face-to-face} if for any two of its cells $c_1,c_2\in T$ one has that
$$
c_1\cap c_2\in(\cF(c_1)\cap\cF(c_2))\cup\{\varnothing\},
$$
that is, the intersection of two cells is either empty or a common face of both cells.
For a face-to-face tessellation $T$ one defines $\cF_k(T)=\bigcup_{c\in T}\cF_{k}(c)$ and $\cF(T)=\bigcup_{c\in T}\cF(c)$.
A face-to-face tessellation in $\RR^{d}$ is called {\bf normal} if each $k$-dimensional face of the tessellation is contained in the boundary of precisely $d+1-k$ cells, for all $k\in\{0,1,\ldots,d-1\}$. 

We denote by $\TT$ the set of all face-to-face tessellations in $\RR^d$, which is supplied with a measurable structure as in \cite[Chapter 10]{SW}. By a \textbf{random tessellation} we understand a particle process $\cT$ in $\RR^d$ (in the usual sense of stochastic geometry, see \cite[Section 4.1]{SW}) satisfying $\supp\cT\in\TT$ almost surely. It is convenient to identify the random point process $\cT$ with its support. A random tessellation is \textbf{stationary}, provided that its distribution is invariant under all shifts in $\RR^d$ and \textbf{isotropic} if its distribution is invariant under all rotations in $\RR^d$. For a stationary random tessellation $\cT$ and $k\in\{0,1,\ldots,d-1\}$ we define the stationary particle process $\cT^{(k)}:=\sum_{F\in \cF_k(\cT)} \delta_{F}$ of $k$-dimensional polytopes, which is referred to as the \textbf{process of $k$-faces}.

Next, we recall the concept of a typical cell (and a typical $k$-face) of a stationary random tessellation $\cT$; see \cite[Section 4.1,4.2]{SW}, \cite[page 450]{SW}, \cite[Section 4.3]{SWGerman} for more details. Let $\cC'$ be the space of non-empty compact subsets of $\RR^d$ endowed with the Hausdorff metric. A \textbf{centre function} is a Borel function $z: \cC' \to \RR^d$ such that $z(C+m) = z(C) + m$ for all $C\in \cC'$ and $m\in \RR^d$. The \textbf{intensity of $k$-faces} of $\cT$  is defined by
$$
\gamma_k(\cT) := \EE \sum_{F\in \cF_k(\cT)} {\bf 1}(z(F) \in [0,1]^{d}),
\qquad
k=0,\ldots, d.
$$
These quantities  are known to be independent of the choice of the centre function $z$. Assuming that $\gamma_k(\cT)\in(0,\infty)$,  the \textbf{typical $k$-face of $\cT$} with respect to the centre function $z$ is the $k$-dimensional random polytope whose distribution is given by
$$
\PP^{z}_{\cT,k}(\,\cdot\,) := {1\over\gamma_k(\cT)} \, \EE\sum_{F\in \cF_k(\cT)} {\bf 1}({F - z(F)} \in\,\cdot\,){\bf 1} (z(F) \in [0,1]^{d}).
$$
In particular, for $k=d$ we get the concept of the \textbf{typical cell} of $\cT$.  It should be noted that translation-invariant characteristics of the distribution of the typical $k$-face do not depend on the choice of $z$. More precisely, if $z$ and $z'$ are two centre functions, then $\PP^{z'}_{\cT, k}$ is the push-forward of $\PP_{\cT,k}^{z}$ under the map $F\mapsto F - z'(F)$.

\section{Construction of $\beta$-, $\beta'$- and Gaussian-Voronoi tessellations}\label{sec:betadefinition}

\subsection{General Laguerre tessellations}

In this section we only briefly recall some facts about Laguerre tessellations and refer the reader to \cite[Sections 3.2--3.4]{GKT20} and \cite[Sections 2.3 and 3.1]{GKT21} for further details.

We start by defining a general Laguerre tessellation. Given two points $v,w \in \RR^{d}$ and $h\in\RR$ we define the \textbf{power} of $w$ with respect to the pair $(v,h)$ as
\[
\pow (w,(v,h)):=\|w-v\|^2+h.
\]
In this situation, $h$ is referred to as the \textbf{weight} (or \textbf{height}) of the point $v$. Let $X$ be a countable set of marked points of the form $(v,h)\in\RR^{d}\times \RR$.
Then the \textbf{Laguerre cell} of $(v,h)\in X$ is the set
\[
C((v,h),X):=\{w\in\RR^{d}\colon \pow(w,(v,h))\leq \pow(w,(v',h'))\text{ for all }(v',h')\in X\}.
\]
The point $v$ is called the nucleus of the cell $C((v,h),X)$. Note that a  Laguerre cell may be empty and even if it is non-empty, it does not need to contain its nucleus. The collection of all non-empty Laguerre cells of $X$  is called the \textbf{Laguerre diagram}:
\[
\cL(X):=\{C((v,h),X)\colon (v,h)\in X, C((v,h),X)\neq\varnothing\}.
\]
In the special case when the heights $h$ of all points are the same (say, $h_0\in\RR$) the above definition leads to the classical \textbf{Voronoi cell}. More precisely, let $Y$ be a countable set of points in $\RR^d$ whose ``marked'' version $X$ is obtained by attaching a fixed weight $h_0$ to each point. Then the Voronoi cell of $v\in Y$ is
\[
V(v,Y)=C((v,h_0), X)=\{w\in\RR^{d}\colon \|w-v\|\leq \|w-v'\| \text{ for all }v'\in Y\}.
\]
The collection of the Voronoi cells of all $v\in Y$ is called the \textbf{Voronoi diagram} $\cV(Y)$. It should be mentioned that a Laguerre diagram is not necessarily a tessellation in $\RR^d$, at least as long as no additional assumptions on the geometric properties of the set $X$ are imposed. Such assumptions have been described in detail in \cite{Ldoc,LZ08,Sch93}. In the present article we are interested in random tessellations built on Poisson point processes. More precisely, we consider a Poisson point process $\xi$ in $\RR^{d}\times E$, where $E\subset \RR$ is a Borel set {(an interval)}, and the corresponding Laguerre diagram $\cL(\xi)$. Lemmas 1 and 2 in \cite{GKT20} (see also \cite[Lemma 2.1]{GKT21}) provide sufficient conditions on $\xi$ which ensure that, almost surely, $\cL(\xi)$ is a stationary random face-to-face normal tessellation in $\RR^d$. In the following we work under these conditions and remark that they are automatically satisfied in the three cases we consider from Section \ref{subsec:ThreeFamilies} on.

\subsection{Laguerre tessellations via paraboloid growth processes} \label{par_growth}

An alternative approach to the construction of Laguerre diagrams uses so-called paraboloid growth processes with overlaps (or simply paraboloid growth process), which were first introduced in \cite{CSY13,SY08} in order to study the asymptotic geometry of random polytopes; see also~\cite{CY_variance_smooth,CYGaussian,CY_var_scaling_random_poly,CY_perturbed}. In this section we briefly describe this rather useful construction and refer for more details to \cite[Section 3.1]{GKT21}.
Let
\[
\Pi_{\pm,x}:=\{(v',h')\in\RR^{d}\times\RR\colon h'=\pm\|v'-v\|^2+h\}
\]
be the upward ($+$) and downward ($-$) standard \textbf{paraboloids} with apex $x:=(v,h)\in\RR^d\times\RR$, denoted as $\apex\Pi_{\pm,x}:=x$. In case $(v,h)=(0,0)$ we simply write $\Pi_{\pm}=\Pi_{\pm,(0,0)}$. Given a set $A\subset \RR^d\times\RR$ we put
\begin{align*}
A^{\downarrow}:&=\{(v,h')\in\RR^{d}\times\RR\colon (v,h) \in A \text{ for some } h\ge h'\},\\
A^{\uparrow}:&=\{(v,h')\in\RR^{d}\times\RR\colon (v,h) \in A \text{ for some } h\leq h'\}.
\end{align*}

Following the definition from \cite{CSY13}, for a given Poisson point process $\xi$ in $\RR^d\times\RR$,
we introduce the \textbf{paraboloid growth process} $\Psi(\xi)$:
$$
\Psi(\xi):=\bigcup\limits_{x\in \xi}\Pi^{\uparrow}_{+,x}.
$$
It should be noted that, in typical situations, the majority of paraboloids will be completely covered by other paraboloids, implying that they do not "contribute" to the model and can thus be omitted without loosing any information about the set $\Psi(\xi)$. This leads to the definition of extreme points. A point $x\in\xi$ is called \textbf{extreme} in the paraboloid growth process $\Psi(\xi)$ if and only if its associated paraboloid is not fully covered by the paraboloids associated with other points of $\xi$, i.e., if
$$
\Pi^{\uparrow}_{+,x}\not\subset \bigcup\limits_{y\in \xi, y\neq x}\Pi^{\uparrow}_{+,y}.
$$
We denote by $\ext(\Psi(\xi))$ the set of all extreme points of the paraboloid growth process $\Psi(\xi)$.
Using the paraboloid growth process we can construct a random diagram in $\RR^{d}$. Given a point $x=(v,h)\in \xi$ define the \textbf{$\Psi$-cell} of $x$ as
$$
C_{\Psi}(x,\xi):=\begin{cases}
\{w\in \RR^{d}\colon \big[(w,0)^{\uparrow}\cup(w,0)^{\downarrow}\big] \cap \bd  \Psi(\xi) \in \Pi_{+,x}\},\qquad &\text{if } x\in\ext(\Psi(\xi)),\\
\varnothing, &\text{otherwise},
\end{cases}
$$
where $\bd A$ denotes the boundary of a set $A$.
In other words, $w$ belongs to $C_{\Psi}(x,\xi)$ if and only if $\|w-v\|^2+h\leq \|w-v'\|+h'$ for all $(v',h')\in\xi$. Thus, the $\Psi$-cell of an extreme point $x$ of the paraboloid growth process $\Psi(\xi)$ is non-empty and coincides with the Laguerre cell $C(x, \xi)$. Next, we construct the diagram $\cL_{\Psi}(\xi)$ as the collection of all non-empty $\Psi$-cells:
$$
\cL_{\Psi}(\xi):=\{C_{\Psi}(x,\xi)\colon C_{\Psi}(x,\xi)\neq\varnothing\}=\{C_{\Psi}(x,\xi)\colon x\in\ext(\Psi(\xi))\}.
$$
We directly have that $\cL_{\Psi}(\xi) = \cL(\xi)$.

\subsection{Three families of random tessellations}\label{subsec:ThreeFamilies}

In this article we consider random tessellations in $\RR^d$ build on the following three families of Poisson point processes. For $\beta>-1$ and $0<\gamma<\infty$ we consider a Poisson point process $\eta_{d,\beta,\gamma}$ in $\RR^{d}\times [0,+\infty)$  whose intensity measure has density
\begin{equation}\label{eq:BetaPoissonIntensity}
(v,h)\mapsto \gamma\,c_{d+1,\beta}h^{\beta},\qquad c_{d+1,\beta}:={\Gamma\left({d+1\over 2}+\beta+1\right)\over \pi^{d+1\over 2}\Gamma(\beta+1)},
\end{equation}
with respect to the Lebesgue measure on $\RR^{d}\times [0,+\infty)$. Further, for $\beta > \frac{d}2 + 1$ and $0<\gamma<\infty$ we consider a Poisson point process $\eta'_{d,\beta,\gamma}$ in $\RR^{d}\times (-\infty,0)$ with intensity measure having density
\begin{equation}\label{eq:BetaPrimePoissonIntensity}
	(v,h)\mapsto \gamma\,c'_{d+1,\beta}(-h)^{-\beta},\qquad c'_{d+1,\beta}:={\Gamma\left(\beta\right)\over \pi^{d+1\over 2}\Gamma(\beta-{d+1\over 2})},
\end{equation}
with respect to the Lebesgue measure on $\RR^{d}\times (-\infty,0)$. The constants $c_{d+1,\beta}$ and $c'_{d+1,\beta}$ in the above definitions are introduced for convenience. For example, they make the statement of Theorem~\ref{tm:betaintersection} below more transparent.
Finally, for $\lambda>0$ and $0<\gamma<\infty$ we consider a Poisson point process $\zeta_{d,\lambda, \gamma}$ in $\RR^{d}\times\RR$ whose intensity measure has density
$$
(v,h)\mapsto \gamma\,e^{\lambda h},
$$
with respect to the Lebesgue measure on $\RR^{d}\times\RR$. It was shown in \cite[Lemma 3]{GKT20} and in \cite[Section 3.3]{GKT21} that the Poisson point processes $\eta_{d,\beta,\gamma}$ and $\zeta_{d,\lambda, \gamma}$ satisfy the sufficient conditions of Lemma 1 and Lemma 2 in \cite{GKT20} and, hence, the corresponding Laguerre diagrams $\cV_{d,\beta,\gamma}:=\cL(\eta_{d,\beta,\gamma})$, $\cV'_{d,\beta,\gamma}:=\cL(\eta'_{d,\beta,\gamma})$ and $\cG_{d,\lambda}:=\cL(\zeta_{d,\lambda,\gamma})$ are stationary random normal tessellations in $\RR^d$, which are called \textbf{$\beta$-Voronoi}, \textbf{$\beta'$-Voronoi} and \textbf{Gaussian-Voronoi tessellations}, respectively. These tessellations have been studied in~\cite{GKT21b,GKT20,GKT21,GKT21a}, where they were considered in  $\RR^{d-1}$ instead of $\RR^d$. Simulations of these tessellations in the plane are shown in Figure \ref{fig:Simulations}. {Note that although the point process $\eta'_{d,\beta,\gamma}$ is well-defined in the range  $\beta>\frac{d+1}{2}$, the corresponding $\beta'$-Voronoi tessellation exists in the smaller range $\beta > \frac d2 + 1$ only, see~\cite[Lemma~3]{GKT20}.}

\begin{figure}[t]
\centering
\includegraphics[width=0.3\columnwidth]{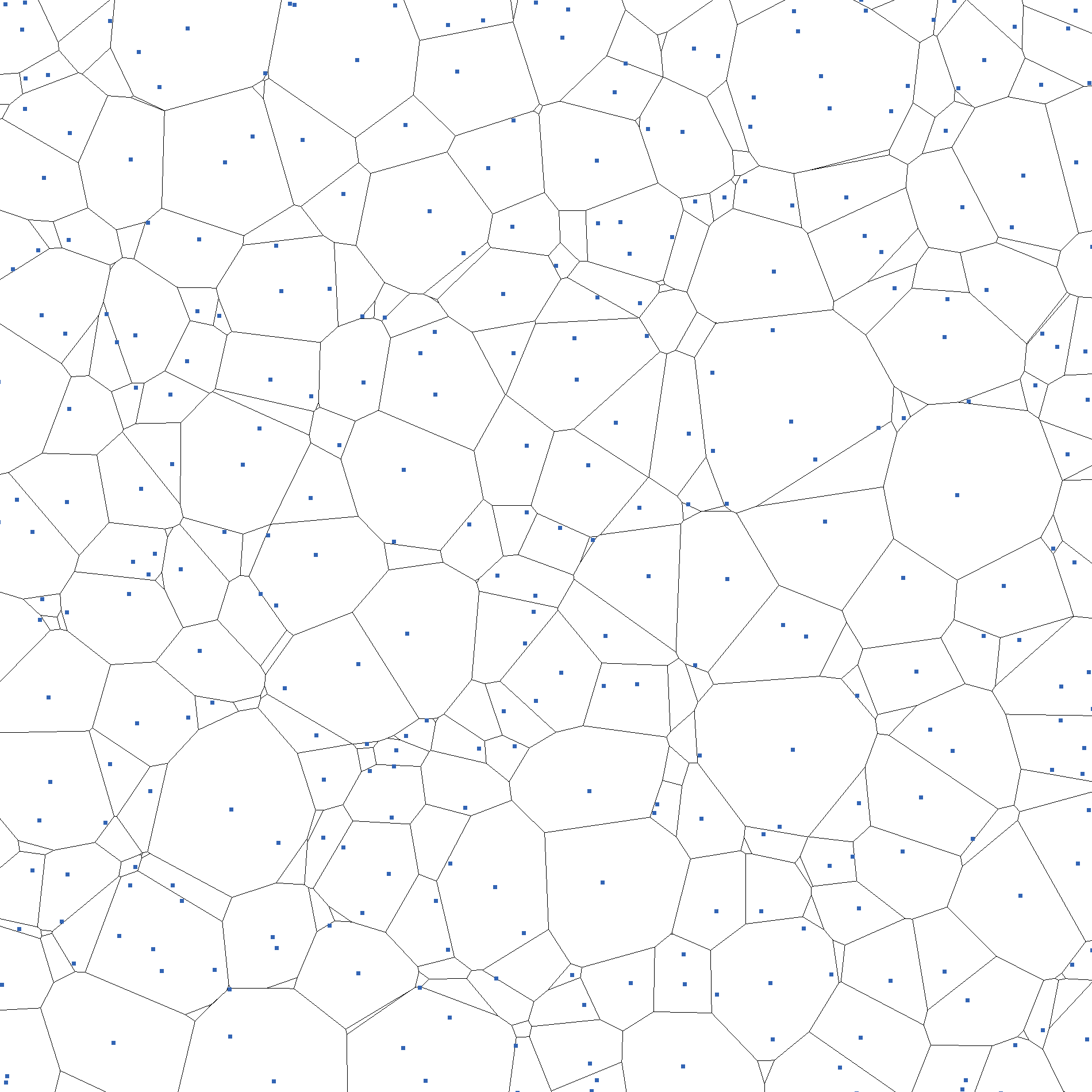}
\quad
\includegraphics[width=0.3\columnwidth]{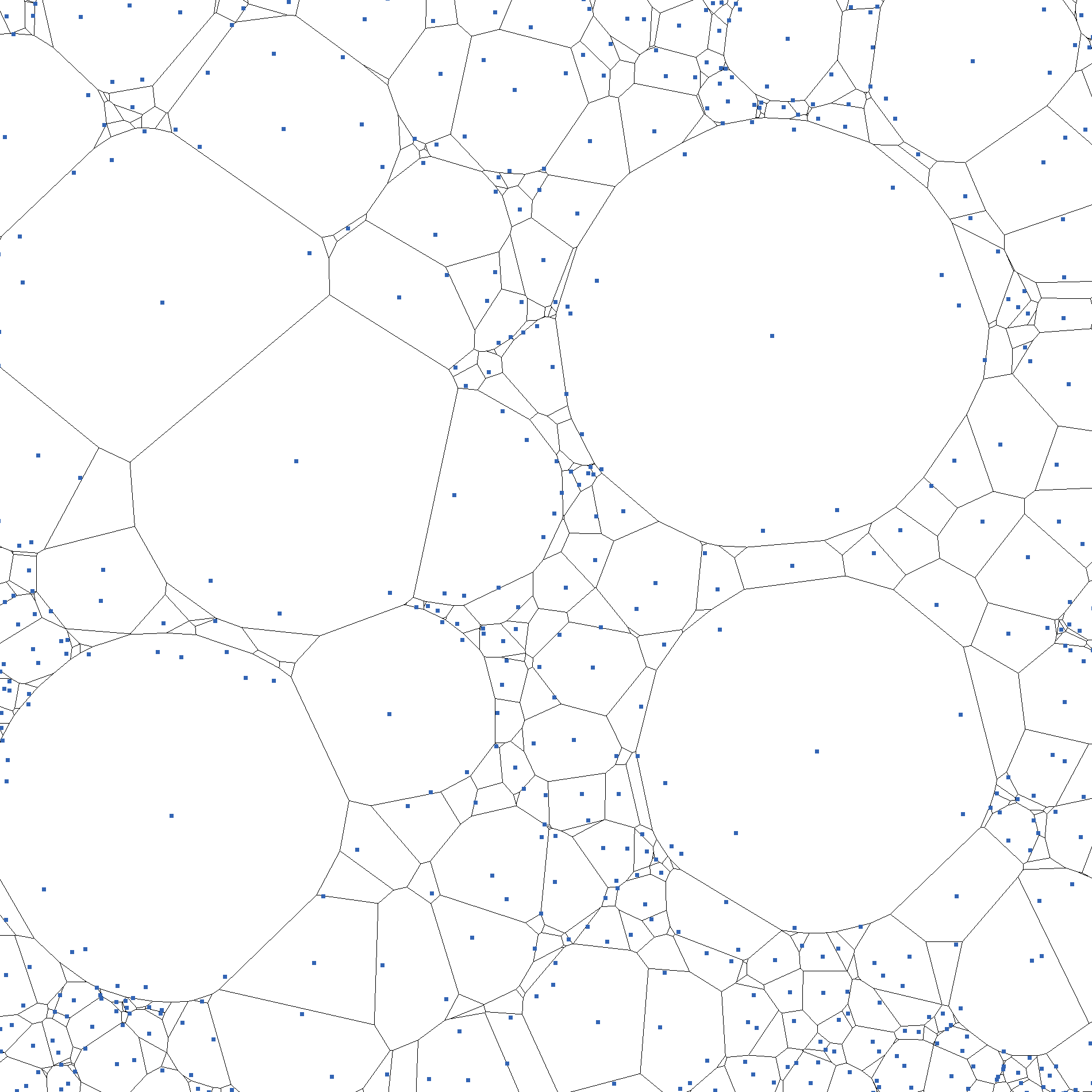}
\quad
\includegraphics[width=0.3\columnwidth]{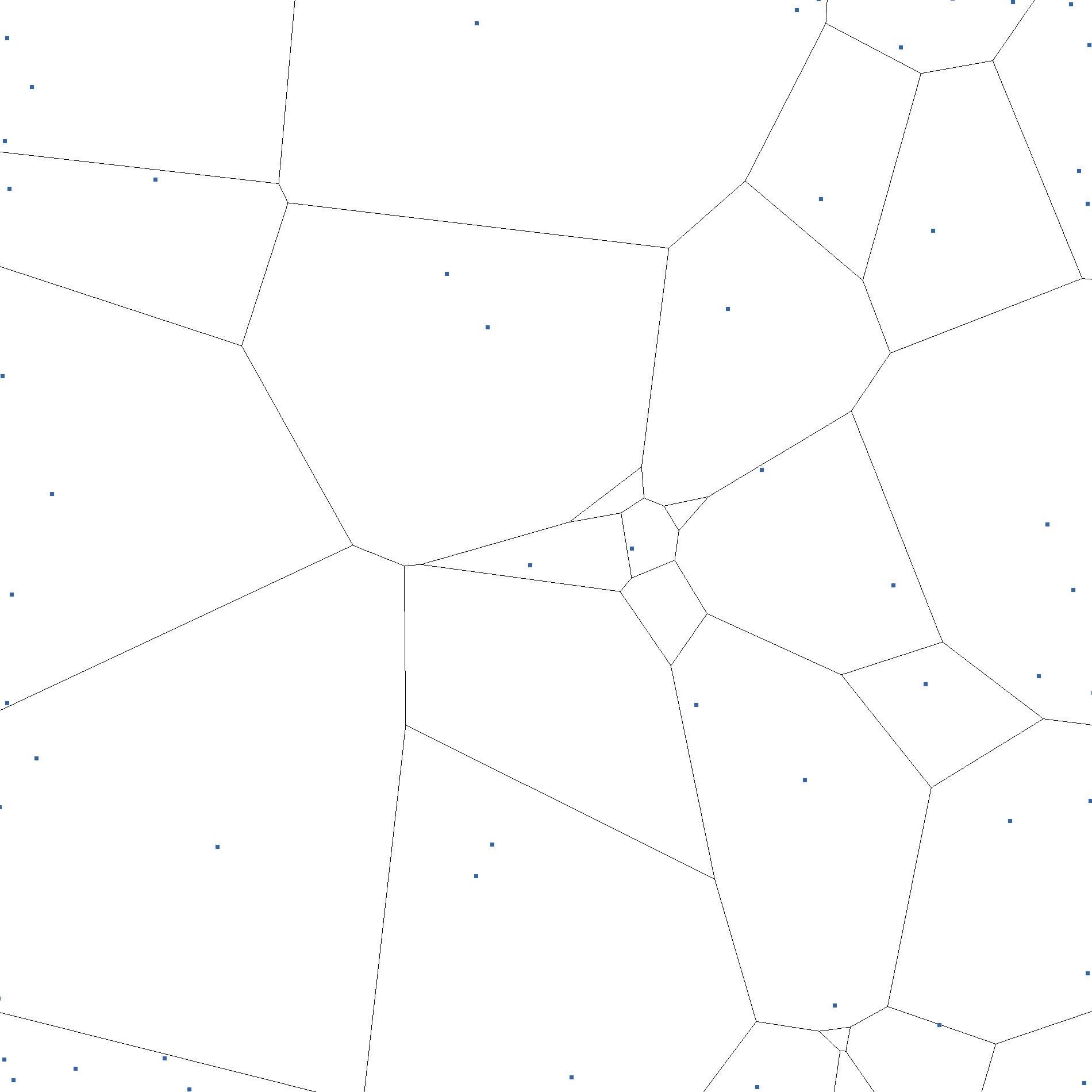}
\caption{Left panel: Simulation of a $\beta$-Voronoi tessellation in $\RR^2$ with $\beta=5$. Middle panel: Simulation of a $\beta'$-Voronoi tessellation in $\RR^2$ with $\beta=2.5$. Right panel: Simulation of a Gaussian-Voronoi tessellation in $\RR^2$.} \label{fig:Simulations}
\end{figure}

\begin{remark}
Note that changing the parameter $\gamma$ amounts to shifting the Poisson point process $\zeta_{d,\lambda, \gamma}$ along the height coordinate $h$. In particular, the distribution of the point process $\cG_{d,\lambda}$ does not depend on the choice of $\gamma$, which is reflected in our notation.
\end{remark}

\begin{remark}\label{rem:convention_beta_-1}
It will be convenient to extend the above definition of the $\beta$-Voronoi tessellation to the case $\beta= -1$ (with arbitrary $\gamma>0$) by
defining $\cV_{d,-1,\gamma}:=\cW_{d, r(d) \gamma}$ to be the classical Poisson-Voronoi tessellation constructed on the homogeneous Poisson point process on $\RR^d$ with constant intensity $r(d) \gamma$, where
\begin{equation*}
r(d) := \Gamma\left({d+1\over 2}\right)\pi^{-{d+1\over 2}}.
\end{equation*}
A justification for this definition is given by the following proposition. For the necessary  background on point processes and weak convergence we refer to~\cite[Chapter~3]{resnick_book}.
\end{remark}

\begin{proposition}
As $\beta\downarrow -1$, the Poisson process $\eta_{d,\beta, \gamma}$ converges, weakly on the space of locally finite integer-valued measures on $\RR^d\times [0,\infty)$, to the Poisson point process whose intensity measure is the Lebesgue measure on $\RR^d \times \{0\}$ times $r(d) \gamma$.
\end{proposition}
\begin{proof}
Write $\beta = -1 + \eps$ with $\eps\downarrow 0$. Then the constant appearing in~\eqref{eq:BetaPoissonIntensity} satisfies
$$
c_{d+1, \beta} ={\Gamma\left({d+1\over 2}+\eps\right)\over \pi^{d+1\over 2}\Gamma(\eps)} = {\Gamma\left({d+1\over 2}\right)\over \pi^{d+1\over 2}}\eps(1+o(1)) ,
$$
as $\eps\downarrow 0$.
It follows that for every $c>0$ and every bounded Borel set $B\subset \RR^d$ the number of points of the Poisson point process $\eta_{d,\beta,\gamma}$ appearing in  $B\times [0,c]$ is Poisson-distributed with expectation
$$
\gamma \lambda_d(B) c_{d+1, \beta}  \int_{0}^c h^{-1+\eps}\dd h \overset{}{\underset{\eps\downarrow 0}\longrightarrow}
\Gamma\left({d+1\over 2}\right) \pi^{- {{d+1}\over 2}} \gamma  \lambda_d(B) = r(d)\gamma\lambda_d(B),
$$
where $\lambda_d$ denotes the $d$-dimensional Lebesgue measure. Since the right-hand side does not depend on $c$, it follows that, for every $0< c_1 < c_2$, the expected number of points of $\eta_{d,\beta,\gamma}$ in $B\times [c_1,c_2]$ converges to $0$, as $\eps \downarrow 0$. Hence, the intensity measure of $\eta_{d,\beta, \gamma}$ converges as $\beta\downarrow -1$ to the Lebesgue measure on $\RR^d \times \{0\}$ times $r(d) \gamma$, vaguely on the space $\RR^d \times [0,\infty)$. Then, the claim of the proposition follows from~\cite[Theorem 16.16 (iv)]{Kallenberg} or~\cite[Propositions~3.6 and 3.19]{resnick_book}.
\end{proof}

\section{Affine sections of $\beta$, $\beta'$- and Gaussian-Voronoi tessellations}

In this section we study the intersection of the $d$-dimensional random tessellations $\cV_{d,\beta,\gamma}$, $\cV_{d,\beta,\gamma}'$ and $\cG_{d,\lambda}$  with an affine subspace $L\subset \RR^d$ of dimension $\ell\in\{1,\ldots,d-1\}$.
By stationarity and isotropy of these tessellations, we may and will assume without loss of generality that $L=\RR^\ell$ is the linear subspace of $\RR^d$ spanned by the first $\ell$ vectors of the standard orthonormal basis of $\RR^d$. The intersection of the tessellation $\cV_{d,\beta,\gamma}$ with $\RR^\ell$ will be denoted by $\cV_{d,\beta,\gamma}\cap \RR^\ell$. Similar convention is used for the tessellations $\cV'_{d,\beta,\gamma}$ and $\cG_{d,\lambda}$.
The following theorem identifies the distribution of $\cV_{d,\beta,\gamma}\cap \RR^\ell$, $\cV'_{d,\beta,\gamma}\cap \RR^\ell$ and $\cG_{d,\lambda}\cap \RR^\ell$.

\begin{theorem}\label{tm:betaintersection}
Fix integers $d\geq 2$ and $1\leq\ell\leq d-1$.
\begin{itemize}
	\item[(i)] For any $\beta\ge -1$ and $\gamma>0$, $\cV_{d,\beta,\gamma}\cap \RR^\ell$  has the same distribution as $\cV_{\ell,\beta+{d-\ell\over 2},\gamma}$.
	\item[(ii)] For any  $\beta>\frac{d}2 + 1$ and $\gamma>0$, $\cV'_{d,\beta,\gamma}\cap \RR^\ell$  has the same distribution as $\cV'_{\ell,\beta-{d-\ell\over 2},\gamma}$.
	\item[(iii)] {For any $\lambda>0$,} $\cG_{d,\lambda}\cap \RR^\ell$  has the same distribution as $\cG_{\ell,\lambda}$.
\end{itemize}
\end{theorem}

Before we move on to the proof of Theorem \ref{tm:betaintersection} we would like to highlight the following special case which deals with sections of the classical Poisson-Voronoi tessellation. Together with the results we obtain below, this fully answers and resolves the problems raised in \cite{ChiuEtAlSection,MilesSection,NagelSurvey,SKM}.

\begin{corollary}\label{cor:VoronoiIntersection}
Fix integers $d\geq 2$ and $1\leq \ell\leq d-1$. Then for any $\rho>0$ the intersection of the $d$-dimensional Poisson-Voronoi tessellation $\cW_{d,\rho}$ of intensity $\rho$ with $\RR^\ell$ has the same distribution as
$$
\cV_{\ell,{d-\ell\over 2}-1,  \pi^{{d+1\over 2}} \rho/\Gamma\left({d+1\over 2}\right)}.
$$
\end{corollary}

\begin{proof}[Proof of Theorem~\ref{tm:betaintersection}]
Let us first consider the case $\ell=d-1$ meaning that we intersect with a hyperplane. Let $\xi$ be one of the Poisson point processes $\eta_{d,\beta,\gamma}$, $\eta'_{d,\beta,\gamma}$ or $\zeta_{d,\lambda,\gamma}$.
The atoms of  $\xi$ live in the space $\RR^{d+1} = \RR^d \times \RR$;  a generic point in this space is denoted by  $(v, h)$ with $v=(v_1,\ldots,v_d)\in \RR^d$ being the spatial coordinate and $h\in \RR$ being the height coordinate. The Laguerre tessellation $\cL(\xi)$ lives in  the space $\RR^d$ defined by the equation $h=0$.  The linear hyperplane $L\equiv \RR^{d-1}\subset \RR^d \subset \RR^{d+1}$ with which our tessellations are intersected is given by the equations $\{v_d=0, h=0\}$.

In order to prove the statement we will use the representation of the Laguerre tessellation $\cL(\xi)$ in terms of the paraboloid growth process $\Psi(\xi)$ as described in Section \ref{par_growth}.
We extend the hyperplane $L \subset \RR^{d}$ by adding the height  coordinate, namely we consider
$$
L':=\{(v,h)=(v_1,\ldots,v_d,h)\in\RR^{d}\times \RR\colon v_d=0\}.
$$
For every point $\bx=(v,h)\in\RR^{d+1}$, the intersection of the $d$-dimensional paraboloid $\Pi_{+,\bx}$  with $L'$ is a $(d-1)$-dimensional upward paraboloid $\Pi_{+,\bx'}\cap L'$ in $L'$ with apex given by
$$
\bx'=f(v,h):=(v_1,\ldots, v_{d-1},0, h+v^2_d)\in L'.
$$
If $\bx$ runs through all atoms of $\xi$, then $\bx'$ runs through all atoms of the point process $f(\xi):=\{f(\bx)\colon \bx\in\xi\}$  on $L'$, which is also a Poisson point process by the mapping theorem; see~\cite[Theorem 5.1]{LP}.
The intersection of the Laguerre tessellation $\cL(\xi)$ with $L$ coincides with the Laguerre tessellation generated by the point process $f(\xi)$ (within $L$).

Now, we are going to identify the {intensity measure $\mu$} of the Poisson point process $f(\xi)$. To this end, we consider the Poisson point processes $\eta_{d,\beta,\gamma}$, $\eta'_{d,\beta,\gamma}$ and $\zeta_{d,\lambda,\gamma}$ separately.

\paragraph{Case (i).} Let first $\beta>-1$. By the mapping theorem for Poisson point processes \cite[Theorem 5.1]{LP}, $f(\eta_{d,\beta,\gamma}):=\{f(x)\colon x\in \eta_{d,\beta,\gamma} \}$ is a Poisson point process in $\RR^{d-1}\times [0,+\infty)\subset L'$.
To compute its intensity measure $\mu$, we take some Borel set $B\subset L$, any $s>0$ and observe that an atom $(v,h)$ of $\eta_{d, \beta,\gamma}$ is mapped  by $f$ to $B\times [0,s]$ if and only if $(v_1,\ldots, v_{d-1}) \in B$ and $h + v_d^2 \leq s$. The latter condition means that $h\leq s$ and $|v_d| \leq \sqrt{s-h}$.  It follows that the intensity measure $\mu$ satisfies
\begin{align}
	\mu(B\times [0,s])&=\gamma\,c_{d+1,\beta}\,\int_{\RR^{d}}\int_{0}^\infty h^{\beta}{\bf 1}(f(v,h)\in B\times [0,s])\,\dd h\dd v\notag\\
	&=\gamma\,c_{d+1,\beta}\,\int_{\RR^{d}}\int_{0}^s h^{\beta}{\bf 1}(v\in B\times [-\sqrt{s-h},\sqrt{s-h}])\,\dd h\dd v\notag\\
	&=2\gamma\,c_{d+1,\beta}\,\lambda_{d-1}(B)\int_{0}^s h^{\beta}\sqrt{s-h}\,\dd h\notag\\
	&=\gamma\,{\Gamma({d+1\over 2}+\beta+1)\over \pi^{d\over 2}\Gamma(\beta+{3\over 2})}\,\lambda_{d-1}(B){s^{\beta+{3\over 2}}\over \beta +{3\over 2}}.\label{eq:case_i_intensity_1}
\end{align}

In the case $\beta = -1$, we let $\eta_{d,-1,\gamma}$ be the Poisson point process on $\RR^d\times \{0\}$ (which is considered as a subset of $\RR^{d+1}$) whose intensity with respect to the Lebesgue measure on $\RR^d\times\{0\}$ is constant and equals  $\Gamma({d+1\over 2})\pi^{-{d+1\over 2}}\gamma$.  Thus, the heights of all points in $\eta_{d,-1,\gamma}$ are $0$. Then,  the Laguerre tessellation generated by $\eta_{d,-1,\gamma}$ on $\RR^d$ coincides with $\cV_{d,-1, \gamma}$ by our convention described in Remark~\ref{rem:convention_beta_-1}. To compute the intensity measure $\mu$ of the Poisson point process $f(\eta_{d,-1,\gamma})$, we take some Borel set $B\subset L$, any $s>0$ and observe that an atom $(v,0)$ of $\eta_{d,-1,\gamma}$ is mapped  by $f$ to $B\times [0,s]$ if and only if $(v_1,\ldots, v_{d-1}) \in B$ and $v_d^2 \leq s$. It follows that the intensity measure $\mu$ satisfies
\begin{align}
\mu(B\times [0,s])
&=
\gamma\,{\Gamma({d+1\over 2})\over \pi^{d+1\over 2}}\,\int_{\RR^{d}}{\bf 1}(f(v,0)\in B\times [0,s])\,\dd v\notag\\
&=
\gamma\,{2\Gamma({d+1\over 2})\over \pi^{d+1\over 2}}\,\lambda_{d-1}(B)\sqrt{s}.\label{eq:case_i_intensity_2}
\end{align}

By differentiating~\eqref{eq:case_i_intensity_1} and~\eqref{eq:case_i_intensity_2} with respect to $s$, it follows that for all $\beta\ge -1$,  the intensity measure of $f(\eta_{d,\beta,\gamma})$  has density
$$
(v',0, h)\mapsto \gamma\,c_{d,\beta+{1\over 2}}h^{\beta+{1\over 2}},
\qquad v'\in \RR^{d-1},\, h>0,
$$
with respect to the Lebesgue measure on $L\times [0,+\infty)$. Consequently, the Laguerre tessellation generated by $f(\eta_{d,\beta,\gamma})$ within $L$ has the same distribution as $\cV_{d-1,\beta+{1\over 2},\gamma}$.

\paragraph{Case (ii).} Next, we deal with $f(\eta'_{d,\beta,\gamma}):=\{f(x)\colon x\in \eta'_{d,\beta,\gamma}\}$.
Let us first consider only those points of $f(\eta'_{d,\beta,\gamma})$ that have negative height and compute the intensity measure $\mu$ of these points. The points with positive height coordinate have no influence on the resulting tessellation, as we will argue below.
To determine $\mu$, we take some Borel set $B\subset L$, any $s<0$ and observe that an atom $(v,-g)$ of $\eta'_{d, \beta,\gamma}$ (with $g>0$) is mapped  by $f$ to $B\times (-\infty, s]$ if and only if $(v_1,\ldots, v_{d-1}) \in B$ and $v_d^2 \leq s + g$. The latter condition means that $g\geq  - s$ and $|v_d| \leq \sqrt{s + g}$.  It follows that the intensity measure $\mu$ satisfies
\begin{align*}
	\mu(B\times (-\infty,s])&=\gamma\,c'_{d+1,\beta}\,\int_{\RR^{d}}\int_{0}^\infty g^{-\beta}{\bf 1}(f(v,-g)\in B\times [-s,\infty))\,\dd g\dd v\\
	&=\gamma\,c'_{d+1,\beta}\,\int_{\RR^{d}}\int_{-s}^\infty g^{-\beta}{\bf 1}(v\in B\times [-\sqrt{s+g},\sqrt{s+g}])\,\dd g\dd v\\
	&=2\gamma\,c'_{d+1,\beta}\,\lambda_{d-1}(B)\int_{-s}^\infty g^{-\beta}\sqrt{g+s}\,\dd g\\
	&=\gamma\,{\Gamma(\beta-{1\over 2})\over \pi^{d\over 2}\Gamma(\beta-{d+1\over 2})}\,\lambda_{d-1}(B){(-s)^{-\beta+{3\over 2}}\over \beta -{3\over 2}}.
\end{align*}
Differentiating, we conclude that the intensity measure $\mu$ has density
$$
(v',0, h)\mapsto \gamma\,c'_{d,\beta+{1\over 2}}(-h)^{-\beta+{1\over 2}},
\qquad v'\in \RR^{d-1},\, h < 0,
$$
with respect to the Lebesgue measure on $L\times (-\infty,0)$. So, by the mapping theorem~\cite[Theorem 5.1]{LP}, 
the restriction of $f(\eta'_{d,\beta,\gamma})$ to $L\times (-\infty,0)$ is a Poisson point process with the same intensity measure 
as $\eta'_{d-1, \beta - \frac 12, \gamma}$. The Laguerre tessellation generated by this Poisson point process 
within $L\equiv \RR^{d-1}$  has the same distribution as $\cV'_{d-1,\beta-{1\over 2},\gamma}$. It remains to observe that adding the points of 
$f(\eta'_{d,\beta,\gamma})$ with positive height coordinate does not change the Laguerre tessellation. Indeed, every point in $\RR^{d-1}\times \{0\}$
is an accumulation point of $\eta'_{d-1, \beta - \frac 12, \gamma}$, hence the lower boundary 
of the paraboloid growth process $\Psi(\eta'_{d-1, \beta - \frac 12, \gamma})$ is contained in $\RR^{d-1}\times (-\infty, 0]$ and points with positive height coordinate have no influence on the tessellation. 

\paragraph{Case (iii).} Finally, we consider $f(\zeta_{d, \lambda, \gamma})$, which is a Poisson point process in $L'$. To compute its intensity measure $\mu$, we take some Borel set $B\subset L$, any $s\in\RR$ and observe that an atom $(v,h)$ of $\zeta_{d,\lambda,\gamma}$ is mapped  by $f$ to $B\times (-\infty,s]$ if and only if $(v_1,\ldots, v_{d-1}) \in B$ and $h + v_d^2 \leq s$. It follows that
\begin{align*}
	\mu(B\times (-\infty,s])&=\gamma\,\int_{\RR^{d}}\int_\RR e^{\lambda h}{\bf 1}(f(v,h)\in B\times (-\infty,s))\,\dd h\dd v\\
	&=2\gamma\,\lambda_{d-1}(B)\int_{-\infty}^s e^{h\lambda}\sqrt{s-h}\, \dd h\\
	&=\gamma \,\lambda_{d-1}(B){\sqrt{\pi} e^{\lambda s}\over \lambda^{3/2}}.
\end{align*}
Thus, the density of the intensity measure of $f(\zeta_{d,\lambda,\gamma})$ is given by
$$
(v',0, h)\mapsto \gamma \sqrt{\pi/\lambda} \, e^{\lambda h}, \qquad v'\in \RR^{d-1},\, h\in \RR.
$$
Hence, the Laguerre tessellation generated by $f(\zeta_{d,\lambda,\gamma})$ within $L$ has the same distribution as $\cG_{d-1,\lambda}$ (recall that the parameter $\gamma$ does not influence the distribution of the Gaussian-Voronoi tessellation).

\vspace*{2mm}
This proves the claim for $\ell=d-1$. For general $1\leq \ell\leq d-2$ we can { inductively} repeat the above argument $d-\ell$ times.
\end{proof}

\section{Sectional Poisson-Voronoi tessellations}\label{sec:sectional_poisson_voronoi_formulas}
\subsection{Face intensities and the expected volume of the typical cell}\label{sec:sectional_poisson_voronoi_formulas1}
As we have shown in Corollary~\ref{cor:VoronoiIntersection}, the sectional Poisson-Voronoi tessellation $\cW_{d,\rho}\cap \RR^\ell$ can be identified with a suitable $\beta$-Voronoi tessellation. This makes it possible to compute explicitly several functionals of the sectional Poisson-Voronoi tessellation. We begin with a formula for the intensity of $j$-dimensional faces. This quantity, denoted by $\gamma_j(\cW_{d,\rho}\cap \RR^\ell)$, has been defined in Section~\ref{sec:preliminaries_tess}.

\begin{theorem}\label{tm:cellintencity}
Let $d\geq 2$, $1\leq \ell\leq d-1$ and $0\leq j\leq \ell$. Then, for any $\rho>0$,
we have
$$
\gamma_j(\cW_{d,\rho}\cap \RR^\ell)
= \rho^{{\ell\over d}}{2\mathbb J_{\ell+1,\ell-j+1}({d-\ell-1\over 2})\pi^{\ell\over 2}\over d(\ell+1)}{\Gamma({(\ell+1)(d-1)\over 2}+1)\Gamma(\ell+1-{\ell\over d})\Gamma({d\over 2}+1)^{\ell+1-{\ell\over d}}\over \Gamma({(\ell+1)(d-1)+1\over 2})\Gamma({\ell+2\over 2})\Gamma({d+1\over 2})^{\ell+1}},
$$
wher
\begin{align}
\mathbb J_{\ell+1,\ell-j+1}\left({d-\ell-1\over 2}\right)
&=
\binom {\ell+1}{j}\,\frac{ \Gamma(\frac{(d-1)(\ell+1)+3}{2}) }{  \sqrt \pi\, \Gamma (\frac{ (d-1)(\ell+1)}2+1)}\,\int_{-\infty}^{+\infty}(\cosh u)^{- (d-1)(\ell+1) - 2}\notag\\
&\hspace{3cm}\times\left(\frac 12  + \ii \frac{ \Gamma({{d+1}\over 2}) }{\sqrt \pi\, \Gamma ({d\over 2})}\,\int_0^u(\cosh v)^{d-1}\dd v \right)^{j} \dd u \label{eq:J_formula}
\end{align}
and $\ii=\sqrt{-1}$ stands for the imaginary unit.
\end{theorem}
\begin{proof}
By Corollary \ref{cor:VoronoiIntersection} we have $\gamma_j(\cW_{d,\rho}\cap \RR^\ell)=\gamma_j(\cV_{\ell,{d-\ell-2\over 2},r})$, where $r=r(d)\rho =\pi^{d+1\over 2}\rho /\Gamma({d+1\over 2})$.  The formula for $\gamma_j(\cV_{\ell,{d-\ell-2\over 2},r})$ can be obtained by combining~\cite[Theorem 6]{GKT20} (which we apply with parameters $d := \ell+1$, $\beta:= {d-\ell-2\over 2}$ and $j:= \ell-j$) with~\cite[Proposition 3]{GKT20} (with parameters $d:= \ell+1$, $k:=\ell+1-j$).
Note that~\cite[Theorem 6]{GKT20} refers to~\cite[Theorem 2]{GKT20} which has to be applied with parameters $d:= \ell+1$,
$s:=1$, $\nu: = 0$, $\gamma: = r$.
\end{proof}

\begin{remark}\label{rm:Jfunction}
The quantities $\mathbb J_{d+1,k}(\beta)$ for general $d\geq 0$, $k\in \{1,\ldots,d+1\}$ and $\beta \ge -1$ have a natural  geometric meaning. Namely, $\mathbb J_{d+1,k}(\beta)$ is equal to the expected sum of internal angles at its $k$-vertex faces of a  random beta-simplex, which is defined  as the convex hull of $d+1$ independent random points with density proportional to $(1-\|x\|^2)^\beta$ in the $d$-dimensional unit ball, see \cite[Section 6.1]{GKT20} and \cite{kabluchko_formula} for details. From this interpretation it directly follows that
\begin{align*}
&\mathbb J_{1,1}(\beta) = \mathbb J_{2,1}(\beta)=\mathbb J_{2,2}(\beta)=1,\\
&\mathbb J_{3,1}(\beta)={1\over 2},\qquad \mathbb J_{3,2}(\beta)={3\over 2},\qquad \mathbb J_{3,3}(\beta)=1,\\
&\mathbb J_{\ell+1,\ell}(\beta)={\ell+1\over 2},\qquad \mathbb J_{\ell+1,\ell+1}(\beta)=1,
\end{align*}
for any $\ell\ge 1$ and $\beta\ge -1$. Moreover, if we denote by $\Sigma_d$ a regular $d$-dimensional simplex and by $\sigma_k(\Sigma_d)$ the internal angle sum at its $k$-vertex faces, then
$$
\mathbb J_{d+1,k}(\infty):=\lim_{\beta\to\infty}\mathbb J_{d+1,k}(\beta)=\sigma_k(\Sigma_d)
$$
according to \cite[Proposition 2]{GKT20}.
\end{remark}

{\setlength{\extrarowheight}{.5em}
\begin{table}[t]
		\centering
		\tiny{
		\begin{tabular}{c|c|c|c|c|c}
			& $d=2$ & $d=3$ & $d=4$ & $d=5$ & $d=6$\\
		\hline
		$\ell=1$ & $\frac{\pi}{4 \sqrt{\rho}}$ & $\frac{\sqrt[3]{3}}{\sqrt[3]{4 \pi \rho} \cdot \Gamma \left(\frac{5}{3}\right)}$ & $\frac{15 \pi ^{3/2}}{64\cdot  \sqrt[4]{8\rho}\cdot \Gamma \left(\frac{3}{4}\right)}$ & $\frac{7 \sqrt[5]{5}}{3 \cdot  \sqrt[5]{648 \pi^2 \rho}\cdot \Gamma \left(\frac{9}{5}\right)}$ & $\frac{2835\cdot  \sqrt[6]{3} \cdot \pi ^{3/2}}{16384\cdot  \sqrt[6]{32 \rho} \cdot \Gamma \left(\frac{5}{6}\right)}$\\
		$\ell=2$ & $-$  & $\frac{5\cdot \sqrt[3]{4}}{\sqrt[3]{3\pi^5\rho^{2}} \cdot  \Gamma \left(\frac{7}{3}\right)}$ & $\frac{24 \sqrt{2}}{35 \sqrt{\pi \rho}}$ & $\frac{77\cdot 2^{4/5}}{5\cdot 15^{3/5} \cdot \pi ^{9/5} \cdot\rho^{2/5} \cdot\Gamma \left(\frac{13}{5}\right)}$ & $\frac{50 \cdot \sqrt[3]{6}}{143 \cdot \sqrt[3]{\rho} \cdot \Gamma \left(\frac{8}{3}\right)}$\\
		$\ell=3$ & $-$ & $-$ & $\frac{280665 \cdot \pi ^{3/2}}{821248 \cdot \sqrt[4]{2} \cdot \rho^{3/4}\cdot  \Gamma \left(\frac{13}{4}\right)}$ & $\frac{56\cdot 15^{3/5} \cdot \sqrt[5]{2/\pi }}{187\cdot  \rho^{3/5} \cdot \Gamma \left(\frac{17}{5}\right)}$ & $\frac{17320875 \cdot \sqrt{3/2}\cdot \pi }{176201728 \cdot\sqrt{\rho}}$\\
		$\ell=4$ & $-$ & $-$ & $-$ & $\frac{144848704\cdot 2^{3/5}}{15^{6/5}  \pi ^{8/5} \left(1692197-141120 \pi ^2\right) \rho^{4/5} \Gamma \left(\frac{21}{5}\right)}$ & $\frac{15\cdot 6^{2/3}}{13 \cdot \rho^{2/3}\cdot \Gamma \left(\frac{13}{3}\right)}$ \\
		$\ell=5$ & $-$ & $-$ & $-$ & $-$ & $\frac{6823504578515625\cdot 3^{5/6} \cdot \pi ^{3/2}}{4912276871446528 \cdot\sqrt[6]{2} \cdot\rho^{5/6} \cdot \Gamma \left(\frac{31}{6}\right)}$
		\end{tabular}
\caption
{$\EE\vol(Z_{d,\ell, \rho})$ for small values of $d$ and $\ell$.}
}
\label{tab:Volume1}
\end{table}}

As a corollary of Theorem~\ref{tm:cellintencity} we can compute the expected volume of the typical cell $Z_{d,\ell, \rho}$ of the sectional Poisson-Voronoi tessellation $\cW_{d,\rho}\cap \RR^\ell$. Note that the volume does not change under shifts, which is why it does not matter how to choose the centre function in the definition of the typical cell.  For $\ell=1$ and $\ell=2$ this quantity has been studied by Miles~\cite{MilesSection} who showed that
\begin{align*}
\EE\vol(Z_{d,1,\rho})&=\rho^{-{1\over d}}{\Gamma(d-{1\over 2})\Gamma({d+1\over 2})^2\over (d-1)!\Gamma(2-{1\over d})\Gamma({d\over 2})\Gamma({d\over 2}+1)^{1-{1\over d}}},\\
\EE\vol(Z_{d,2,\rho})&=\rho^{-{2\over d}}{3d\cdot \Gamma({3d\over 2}-1)\Gamma({d+1\over 2})^3\over \pi\Gamma({3d-1\over 2})\Gamma(3-{2\over d})\Gamma({d\over 2}+1)^{3-{2\over d}}},
\end{align*}
see Formulas (4.1) and (4.4) in \cite{MilesSection}. Our result generalizes this to arbitrary $1\leq \ell\leq d-1$; special cases with small values of $d$ and $\ell$ are summarized in Table \ref{tab:Volume1}.

\begin{corollary}\label{cor:volumeintersection}
Let $\rho>0$. Then, for any $d\geq 2$ and $1\leq \ell\leq d-1$ we have
\begin{align*}
\EE\vol(Z_{d,\ell,\rho})
&=\rho^{-{\ell\over d}}{d(\ell+1)\over 2\mathbb J_{\ell+1,1}({d-\ell-1\over 2})\pi^{\ell\over 2}}{\Gamma({(\ell+1)(d-1)+1\over 2})\over \Gamma({(\ell+1)(d-1)\over 2}+1)}{\Gamma({\ell+2\over 2})\over \Gamma(\ell+1-{\ell\over d})}{\Gamma({d+1\over 2})^{\ell+1}\over \Gamma({d\over 2}+1)^{\ell+1-{\ell\over d}}}.
\end{align*}
\end{corollary}
\begin{proof}
It is known from~\cite[Equation (10.4)]{SW} that $\EE\vol(Z_{d,\ell, \rho})=\gamma_\ell(\cW_{d,\rho}\cap \RR^\ell)^{-1}$. The right-hand side is known from Theorem \ref{tm:cellintencity}.
\end{proof}

\begin{remark}
Corollary~\ref{cor:volumeintersection} stays true for $\ell= d$ where it gives the expected volume of a  typical Poisson-Voronoi cell to be $\EE\vol(Z_{d,d,\rho}) = 1/\rho$. The quantity $\mathbb J_{\ell+1,1}(-\frac 12)$ cancels with the Gamma-factors by the formula given in~\cite[Theorem~3.9]{kabluchko_recursive_scheme} and the Legendre duplication formula for the Gamma function. Theorem~\ref{tm:cellintencity} also stays true for $\ell=d$ and gives the intensity of $j$-faces in the Poisson-Voronoi tessellation; see~\cite[Remark~2.10]{kabluchko_formula} for another formula.
\end{remark}

In the next result we compute the limit of the intensity of $j$-dimensional faces in the $d$-dimensional Poisson-Voronoi tessellation intersected with $\RR^\ell$ in the regime when $d\to\infty$ while $\ell\in \NN$ stays fixed.

\begin{proposition}\label{prop:lim_gamma_j_d_to_infty}
Fix some $\ell \in \NN$ and $0\leq j\leq \ell$. Let $(\rho_d)_{d\in\NN}$ be a positive sequence  such that $\lim\limits_{d\to\infty}(\rho_d)^{1/d}=\kappa>0$.   Then,
$$
\lim_{d\to\infty} \gamma_j(\cW_{d,\rho_d}\cap \RR^\ell)
=
\frac{\mathbb J_{\ell+1,\ell - j + 1}(\infty)(\kappa^2\pi e)^{\ell\over 2}} {\sqrt{\ell+1}}
\frac {2(\ell-1)!} {\Gamma({\ell\over 2})},
$$
where $\mathbb J_{\ell+1,\ell-j+1}(\infty)$ is the sum of angles at {$(\ell-j)$-dimensional faces} of a regular $\ell$-dimensional simplex $\Sigma_\ell$; see Remark \ref{rm:Jfunction}.
\end{proposition}
\begin{remark}
For example, we may take $\rho_d= \rho>0$ to be constant, in which case $\kappa=1$.
\end{remark}
\begin{proof}[Proof of Proposition~\ref{prop:lim_gamma_j_d_to_infty}]
By Theorem~\ref{tm:cellintencity},
\begin{align*}
\lim_{d\to\infty}\gamma_j(\cW_{d,\rho_d}\cap \RR^\ell)
&=
\lim_{d\to\infty}
\frac {2\mathbb J_{\ell+1,\ell - j +1}({d-\ell-1\over 2})\pi^{\ell\over 2}} {d(\ell+1)\rho_d^{-{\ell\over d}}}
\frac {\Gamma({(\ell+1)(d-1)\over 2}+1)} {\Gamma({(\ell+1)(d-1)+1\over 2})}
\frac {\Gamma(\ell+1-{\ell\over d})}{\Gamma({\ell+2\over 2})}
\frac {\Gamma({d\over 2}+1)^{\ell+1-{\ell\over d}}}{\Gamma({d+1\over 2})^{\ell+1}}
\\
&=
\frac {2\mathbb J_{\ell+1,\ell - j +1}(\infty)(\kappa^2\pi)^{\ell\over 2}}{\ell+1}
\frac {2\Gamma(\ell)} {\Gamma({\ell\over 2})}
\lim_{d\to\infty}
\frac 1d
\frac{\Gamma({(\ell+1)(d-1)\over 2}+1)} {\Gamma({(\ell+1)(d-1)+1\over 2})}
\frac{\Gamma({d\over 2}+1)^{\ell+1-{\ell\over d}}}{\Gamma({d+1\over 2})^{\ell+1}}.
\end{align*}
By Stirling's formula for the Gamma function, $\Gamma(z)=\sqrt{2\pi/z}(z/e)^z(1+O(z^{-1}))$.  Since $\lim\limits_{n\to\infty}{\Gamma(n)n^z\over\Gamma(n+z)}=1$ we get
$$
\lim_{d\to\infty}\gamma_j(\cW_{d,\rho_d}\cap \RR^\ell)
=
\frac {\mathbb J_{\ell+1,\ell - j +1}(\infty)(\kappa^2\pi e)^{\ell\over 2}}{\sqrt{\ell+1}}
\frac {2\Gamma(\ell)} {\Gamma({\ell\over 2})}
\lim_{d\to\infty}
\Big(\frac d {4\pi}\Big)^{\ell\over 2d}
=
\frac {\mathbb J_{\ell+1,\ell - j +1}(\infty)(\kappa^2\pi e)^{\ell\over 2}}{\sqrt{\ell+1}}
\frac {2\Gamma(\ell)} {\Gamma({\ell\over 2})}.
$$
This completes the argument.
\end{proof}

We now study  the asymptotic behaviour of the expected volume of the typical cell in the sections of fixed dimension $\ell$ of a high-dimensional  Poisson-Voronoi tessellation.

\begin{corollary}\label{cor:limit}
Let $(\rho_d)_{d\in\NN}$ be a positive sequence  such that $\lim\limits_{d\to\infty}(\rho_d)^{1/d}=\kappa>0$. Then, for every $\ell\in \NN$,
\begin{align*}
\lim_{d\to\infty}\EE\vol(Z_{d,\ell,\rho_d})&={\sqrt{\ell+1}\over \mathbb J_{\ell+1,1}(\infty)(\kappa^2\pi e)^{\ell\over 2}}{\Gamma({\ell\over 2})\over 2(\ell-1)!},
\end{align*}
where $\mathbb J_{\ell+1,1}(\infty)$ is the sum of solid angles of the regular $\ell$-dimensional simplex $\Sigma_\ell$ at its vertices; see Remark~\ref{rm:Jfunction}.
\end{corollary}
\begin{proof}
This follows from the fact that $\EE\vol(Z_{d,\ell,\rho_d})=\gamma_\ell(\cW_{d,\rho_d}\cap \RR^\ell)^{-1}$ (see \cite[Equation (10.4)]{SW}) by applying Proposition~\ref{prop:lim_gamma_j_d_to_infty} with $j=\ell$.
\end{proof}

In the special cases $\ell =1,2$ and for every fixed $\rho>0$, Corollary \ref{cor:limit} combined with the results of Remark \ref{rm:Jfunction} yields, for any constant $\rho>0$,  the  limit relations
\begin{align*}
\lim_{d\to\infty}\EE\vol(Z_{d,1,\rho}) = {1\over \sqrt{2e}},
\qquad\text{and}\qquad
\lim_{d\to\infty}\EE\vol(Z_{d,2,\rho}) = {\sqrt{3}\over e\pi},
\end{align*}
which were already known from the work of Miles~\cite[pp.~318, 319]{MilesSection}. Moreover, for $\ell= 3$ we get
\begin{align*}
\lim_{d\to\infty}\EE\vol(Z_{d,3, \rho})&=\big(4e^{3/2}(3\arccos(1/3)-\pi)\big)^{-1},
\end{align*}
for example. This follows from the fact that the solid angle at a vertex of a regular tetrahedron is $\theta:={1\over 4\pi}(3\arccos(1/3)-\pi)$, implying that $\JJ_{4,1}(\infty)=\sigma_1(\Sigma_3)=4\theta={3\over\pi}\arccos(1/3)-1$.

\subsection{Expected intrinsic volumes and $f$-vectors of  typical $k$-faces}\label{sec:sectional_poisson_voronoi_formulas2}

Together with the volume of the typical cell $Z_{d,\ell, \rho}$ of the sectional Poisson-Voronoi tessellation we can consider its intrinsic volumes. We {recall from~\cite[p.~222]{SW}} that the intrinsic volume $V_m(K)$ of order $0\leq m\leq d$ of a compact convex set $K\subset\RR^d$ may be defined as
$$
V_m(K) := {d!\over m!(d-m)!}{\Gamma({m\over 2}+1)\Gamma({d-m\over 2}+1)\over \Gamma({d\over 2}+1)}\EE\lambda_m(K|L),
$$
where $L\subset\RR^d$ is a uniformly distributed random subspace of dimension $m$, $K|L$ denotes the orthogonal projection of $K$ onto $L$ and $\lambda_m(K|L)$ its $m$-dimensional Lebesgue measure.  In addition, instead of the typical sectional cell we can consider for $1\leq k\leq\ell$ the typical $k$-face $Z_{d,\ell, \rho}^{(k)}$ of the sectional Poisson-Voronoi tessellation $\cW_{d,\rho}\cap\RR^\ell$, see Section~\ref{sec:preliminaries_tess} or~\cite[page 450]{SW} for a formal definition. For example, for $k=\ell$ we get back the typical cell, for $k=\ell-1$ the typical facet and for $k=1$ the typical edge of the sectional tessellation. Using the results from \cite[pages 466-467]{SW} for general stationary and isotropic random tessellations we conclude (by combining the last two formulas there) that
\begin{equation}\label{eq:E_V_j_typical_k_face_formula}
\EE V_j(Z_{d,\ell, \rho}^{(k)}) = {\ell!\over j!(\ell-j!)}{\Gamma({j\over 2}+1)\Gamma({\ell-j\over 2}+1)\over\Gamma({\ell\over 2}+1)}{\gamma_{k-j}(\cW_{d,\rho}\cap \RR^{\ell-j})\over\gamma_{k}(\cW_{d,\rho}\cap \RR^\ell)},
\end{equation}
where $d\geq 2$, and $1\leq\ell\leq d-1$, $0\leq k\leq \ell$ and $0\leq j\leq k$.
This expression can be made fully explicit in view of Corollary \ref{cor:volumeintersection}:
\begin{multline}
\EE V_j(Z_{d,\ell, \rho}^{(k)}) = \Big({\rho\over \Gamma({d\over 2}+1)}\Big)^{-{j\over d}}{(\ell+1)!\over j!(\ell-j+1)!}{\Gamma({j\over 2}+1)\over \pi^{j/2}}
{\JJ_{\ell-j+1,\ell-k+1}({d-\ell+j-1\over 2})\over\JJ_{\ell+1,\ell-k+1}({d-\ell-1\over 2})}\\
\times
{\Gamma({(\ell-j+1)(d-1)\over 2}+1)\Gamma(\ell-j+1-{\ell-j\over d})\over\Gamma({(\ell+1)(d-1)\over 2}+1)\Gamma(\ell+1-{\ell\over d})}
{\Gamma({(\ell+1)(d-1)+1\over 2})\Gamma({d+1\over 2})^j\over \Gamma({(\ell-j+1)(d-1)+1\over 2})\Gamma({d\over 2}+1)^{j}}.
\end{multline}
For intersections of dimension $\ell=2$, Miles~\cite[Equations (4.4), (4.5) on p.~319]{MilesSection} derived a formula for the expected area and perimeter of the typical cell which are particular cases of the above formula. Using Proposition~\ref{prop:lim_gamma_j_d_to_infty} it is easy to derive the large $d$ limit of~\eqref{eq:E_V_j_typical_k_face_formula}. Namely, if $\lim\limits_{d\to\infty}(\rho_d)^{1/d}=\kappa>0$, then
$$
\lim_{d\to\infty} \EE V_j(Z_{d,\ell, \rho_d}^{(k)})
=
\frac{\sqrt{\ell+1}}{\sqrt{\ell-j+1}}
\cdot
\frac{\Gamma(\frac j2 + 1)}{j! (\kappa^2 \pi e)^{j/2}}
\cdot
\frac{\mathbb J_{\ell-j+1,\ell-k+1}(\infty)}{\mathbb J_{\ell+1,\ell-k+1}(\infty)}.
$$
For intersections of dimensions $\ell = 2$ and $3$ (and $k=\ell$) we recover results of Miles~\cite[pp.~319, 320]{MilesSection}.

Finally, we deal with the expected number of $j$-dimensional faces of the typical cell of the sectional Voronoi tessellation $\cW_{d,\rho}\cap\RR^\ell$, which we denote by $\EE f_j(Z_{d,\ell, \rho})$, for $\rho>0$, $d\geq 2$, $1\leq\ell\leq d-1$ and $0\leq j\leq\ell-1$. Using the fact that, with probability $1$, each $j$-dimensional face of the sectional tessellation is contained in the boundary of precisely $\ell-j+1$ of its cells (by normality of the tessellation), it follows that
$$
\EE f_j(Z_{d,\ell, \rho}) = (\ell-j+1){\gamma_j(\cW_{d,\rho}\cap \RR^\ell)\over \gamma_\ell(\cW_{d,\rho}\cap \RR^\ell)}.
$$
We can now apply Corollary \ref{cor:volumeintersection} to conclude that
$$
\EE f_j(Z_{d,\ell, \rho}) = (\ell-j+1){\JJ_{\ell+1,\ell-j+1}({d-\ell-1\over 2})\over\JJ_{\ell+1,1}({d-\ell-1\over 2})},
$$
independently of $\rho$. Clearly, $\EE f_0(Z_{d,1, \rho})=2$ for any $d\geq 2$. Also, $\EE f_j(Z_{d,2, \rho})=6$ for any $d\geq 2$ and $0\leq j\leq 1$, since the sectional Voronoi tessellation is stationary and normal. Some non-trivial values for space dimensions $d=4,5,6$ are collected in Table~\ref{tab:Fvector}.

{\setlength{\extrarowheight}{.5em}
\begin{table}[t]
\centering
\small{
\begin{tabular}{c|c|c|c|c|c|c}
& $d=4$, $\ell=3$ & $d=5$, $\ell=3$ & $d=5$, $\ell=4$ & $d=6$, $\ell=3$ & $d=6$, $\ell=4$ & $d=6$, $\ell=5$\\
\hline
$j=0$ & $\frac{10\,240}{401}$ & $\frac{67\,200 \pi ^2}{26\,741}$ & $\frac{4\,233\,600 \pi ^2}{1\,692\,197-141\,120 \pi ^2}$ & $\frac{524\,288}{21\,509}$ & $\frac{52\,003}{400}$ &$\frac{34\,394\,098\,106\,368}{37\,477\,698\,299}$ \\
$j=1$ & $\frac{15\,360}{401}$ & $\frac{100\,800 \pi ^2}{26\,741}$ & $\frac{8\,467\,200 \pi ^2}{1\,692\,197-141\,120 \pi ^2}$ & $\frac{786\,432}{21\,509}$ & $\frac{52\,003}{200}$ & $\frac{85\,985\,245\,265\,920}{37\,477\,698\,299}$ \\
$j=2$ & $\frac{5\,922}{401}$ & $2+\frac{33\,600 \pi ^2}{26\,741}$ & $\frac{10\,153\,182+4\,233\,600 \pi ^2}{1\,692\,197-141\,120 \pi ^2}$ &  $\frac{305\,162}{21\,509}$ & $\frac{162\,009}{1\,000}$ & $\frac{74\,276\,903\,321\,600}{37\,477\,698\,299}$\\
$j=3$ & $-$ & $-$ & $\frac{10\,153\,182}{1\,692\,197-141\,120 \pi ^2}$ & $-$ & $\frac{64\,003}{2\,000}$ & $\frac{25\,430\,109\,716\,480}{37\,477\,698\,299}$\\
$j=4$ & $-$ & $-$ & $-$ & $-$ & $-$ & $\frac{53\,194\,508\,510}{707\,126\,383}$
\end{tabular}
\caption{{$\EE f_j(Z_{d,\ell, \rho})$ for small values of $d$, $\ell$ and $j$.}}}
\label{tab:Fvector}
\end{table}}

Similarly, we can compute the expected number of $j$-dimensional faces of the typical $k$-dimensional face  $Z_{d,\ell, \rho}^{(k)}$ of the sectional Poisson-Voronoi tessellation $\cW_{d,\rho}\cap\RR^\ell$ for $d\geq 2$, $1\leq\ell\leq d-1$, $1\leq k\leq\ell$ and $0\leq j\leq k-1$:
$$
\EE f_j(Z_{d,\ell, \rho}^{(k)}) = (k-j+1){\JJ_{\ell+1,\ell-j+1}({d-\ell-1\over 2})\over\JJ_{\ell+1,\ell-k+1}({d-\ell-1\over 2})}.
$$
In the large $d$ limit this becomes
$$
\lim_{d\to\infty} \EE f_j(Z_{d,\ell, \rho}^{(k)}) = (k-j+1){\JJ_{\ell+1,\ell-j+1}(\infty)\over\JJ_{\ell+1,\ell-k+1}(\infty)}.
$$
Again, for $\ell=2,3$ we recover results of Miles~\cite[p.~320]{MilesSection}.

\section{Convergence to the Gaussian-Voronoi tessellation in high dimensions}

In Sections~\ref{sec:sectional_poisson_voronoi_formulas1} and \ref{sec:sectional_poisson_voronoi_formulas2} we computed explicitly several characteristics of the sectional Poisson-Voronoi tessellation $\cW_{d,\rho_d}\cap \RR^\ell$ and the limits of these characteristics in the regime when $d\to\infty$ and $(\rho_d)^{1/d} \to \kappa>0$, while $\ell\in \NN$ stays fixed. It turns out that these limits coincide with the corresponding characteristics of the tessellation $\cG_{\ell,\lambda}$ with $\lambda= \kappa^2 \pi e$.
For example, for the typical cell $Z(\cG_{\ell,\lambda})$ of the Gaussian-Voronoi tessellation $\cG_{\ell,\lambda}$ it is known from~\cite[Section 5]{GKT21} (where the special case $\lambda=1/2$ has been considered)  that
$$
\EE\vol (Z(\cG_{\ell,\lambda}))={\sqrt{\ell+1}\over \JJ_{\ell+1,1}(\infty)\lambda^{\ell\over 2}}{\Gamma({\ell\over 2})\over 2(\ell-1)!}.
$$
This formula coincides with the one obtained in Corollary \ref{cor:limit} if we choose $\lambda = \kappa^2 \pi e$.
In the next two theorems we explain this and other similar coincidences by proving weak convergence of the corresponding tessellations and the typical cells.

\begin{theorem}\label{theo:conv_tess}
Take any positive sequence $(\rho_d)_{d\in\NN}$ with $\lim_{d\to\infty}(\rho_d)^{1/d}=\kappa>0$ and let $\ell\in \NN$ be fixed.  Then, as $d\to\infty$, the sectional Poisson-Voronoi tessellation  $\cW_{d,\rho_d}\cap \RR^\ell$ converges to $\cG_{\ell, \kappa^2\pi e}$ in the following sense: It is possible to define all these random tessellations on the same probability space such that for every 
$\ell$-dimensional ball $B_R\subset \RR^\ell$ of radius $R>0$ centred at the origin the probability that the restrictions of $\cW_{d,\rho_d}\cap \RR^\ell$ and $\cG_{\ell, \kappa^2\pi e}$ to $B_R$ coincide, converges to $1$,  as $d\to\infty$.
\end{theorem}
\begin{remark}
The skeleton of a (random) tessellation $\cT$ is the (random) closed set $\skel(\cT)=\bigcup_{c \in \cT}\bd c$, where $\bd c$ denotes the topological boundary of the cell $c$. The  mode of convergence appearing in Theorem~\ref{theo:conv_tess} implies that the random closed set  $\skel (\cW_{d,\rho_d}\cap \RR^\ell)$ converges to the random closed set $\skel(\cG_{\ell, \kappa^2\pi e})$ weakly as $d\to\infty$; see \cite[Chapter~2]{SW} for this concept.
\end{remark}
\begin{proof}[Proof of Theorem~\ref{theo:conv_tess}]
The proof of this theorem basically follows the same route as the proof of Theorem~4.2 in \cite{GKT21}, which is the reason why we leave out some details here. By Corollary~\ref{cor:VoronoiIntersection}, the sectional tessellation $\cW_{d,\rho_d}\cap \RR^\ell$ has the same distribution as $\cV_{\ell,\beta_d,\gamma_d}$ with
$$
\beta_d= \frac 12 (d-\ell) -1
\qquad\text{and}\qquad
\gamma_d=\pi^{d+1\over 2}\Gamma\left({d+1\over 2}\right)^{-1}\rho_d.
$$
We will prove that, as $d\to\infty$, the Poisson point processes $\eta_{\ell, \beta_d,\gamma_d}$ converge, after an appropriate vertical shift, to $\zeta_{\ell,\kappa^2\pi e,1}$; see Section~\ref{subsec:ThreeFamilies} for their definitions. The vertical shift $Q_{d}: \RR^{\ell+1}\to \RR^{\ell+1}$ is given by
$Q_d(v,h)= (v, h-a_d)$ with
\begin{equation}\label{eq:ad}
a_d={1\over \pi}\Big({\pi\Gamma({d-\ell\over 2})\over \rho_d}\Big)^{2\over d-\ell-2}.
\end{equation}
Note that applying such a vertical shift to a point process does not change the resulting Laguerre tessellation since it amounts to shifting all paraboloids along the height coordinate.

In the following, we show that the intensity function  of the Poisson point process $\xi_{\ell,d}:=Q_{d}(\eta_{\ell, \beta_d, \gamma_d})$ converges, as $d\to\infty$, to the intensity function of $\xi_{\ell,\infty}:=\zeta_{\ell,\kappa^2\pi e,1}$  uniformly on every compact subset of $\RR^{\ell+1}$. Indeed, the intensity function of $Q_{d}(\eta_{\ell, \beta_d, \gamma_d})$ is given by
\begin{align*}
f_d(v,h)
&=
{\rho_d\pi^{d-\ell\over 2}\over \Gamma({d-\ell\over 2})} (h+a_d)^{{d-\ell-2\over 2}} {\bf 1} \big( h+a_d>0\big)\\
&=
{\rho_d\pi^{d-\ell\over 2}\over \Gamma({d-\ell\over 2})}a_d^{d-\ell-2\over 2}
\Big(1+{h\over a_d}\Big)^{{d-\ell-2\over 2}}{\bf 1} \big( h+a_d>0\big)
\\
&=
\Big(1+{h\over a_d}\Big)^{a_d\cdot{d-\ell-2\over 2a_d}}{\bf 1} \big( h+a_d>0\big).
\end{align*}
Stirling's formula for the Gamma function and~\eqref{eq:ad} yield
\begin{align*}
\lim_{d\to\infty}{d-\ell-2\over 2a_d}&={1\over \kappa^2\pi}\lim_{d\to\infty}{{d-\ell+2}\over 2\Gamma({d-\ell\over 2})^{2\over d-\ell+2}}=\kappa^2\pi e\lim_{d\to\infty}{{d-\ell+2}\over (d-\ell)^{d-\ell\over d-\ell+2}}=\kappa^2\pi e.
\end{align*}
Note that, in particular, $a_d\to\infty$ as $d\to\infty$. We conclude that  $\lim_{d\to\infty} f_d(v,h) =e^{\kappa^2\pi e h}$ uniformly as long as $h$ stays bounded. {By standard results~\cite[Propositions~3.6 and 3.19]{resnick_book}, this also implies weak convergence of the corresponding Poisson point processes.}

After we have shown the convergence of the  point processes $\xi_{\ell,d}$ to $\xi_{\ell,\infty}$, as $d\to\infty$, we explain the procedure allowing to transfer this result to the convergence of the corresponding tessellations $\cL_{\Psi}(\xi_{\ell,d})$ to $\cL_{\Psi}(\xi_{\ell,\infty})$ as $d\to\infty$. Note that $\cL_{\Psi}(\xi_{\ell,d})$ has the same distribution as $\cW_{d,\rho_d}\cap \RR^\ell$ and $\cL_{\Psi}(\xi_{\ell,\infty})$ has the same distribution as $\cG_{\ell, \kappa^2\pi e}$ (see  Corollary~\ref{cor:VoronoiIntersection} and Section~\ref{par_growth}).

We fix an $\ell$-dimensional ball $B_R\subset \RR^\ell$ of radius $R>0$ centred at the origin and for any $\varepsilon>0$ we aim to find a region $K(R,\varepsilon)\subset \RR^{\ell+1}$, independent of $d$, such that with probability at least $1-\varepsilon$ the restrictions of the tessellations $\cL_{\Psi}(\xi_{\ell,d})$ and $\cL_{\Psi}(\xi_{\ell,\infty})$ to $B_R$ are completely determined by the restrictions of the point processes $\xi_{\ell,d}$ and $\xi_{\ell,\infty}$ to $K(R,\varepsilon)$ for any $d$. To this end, we note that $\cL_{\Psi}(\xi_{\ell,d})$ may be regarded as a vertical projection along the $h$-axis of the boundary of the corresponding paraboloid growth process $\Psi(\xi_{\ell,d})$ (see Section \ref{par_growth}). From this it follows that if the restrictions of the tessellations $\cL_{\Psi}(\xi_{\ell,d})$ and $\cL_{\Psi}(\xi_{\ell,\infty})$ to $B_R$ do not coincide, then the boundaries of the corresponding paraboloid hull processes $\bd\Psi(\xi_{\ell,d})$ and  $\bd\Psi(\xi_{\ell, \infty})$ restricted to the cylinder $B_R\times \RR$ do not coincide as well. The construction of the region $K(R,\varepsilon)$ can be now performed as follows. First, we consider the event $E(T,r)$ that $\bd\Psi(\xi_{\ell,d})$ restricted to the cylinder $B_R\times \RR$ is completely determined by the restriction of $\bd\Psi(\xi_{\ell,d})$ to the cylinder $B_{R+r}\times (-\infty, T]$ for some $T,r>0$. By this we mean that for every paraboloid $\Pi_{+,x}$ with $x\in\ext(\Psi(\xi_{\ell,d}))$ the set $\Pi_{+,x}\cap \bd\Psi(\xi_{\ell,d})$  either  does not intersect $B_R\times \RR$ or is included in $B_{R+r}\times (-\infty, T]$. We have that
\begin{equation}\label{eq:151122a}
1-\PP(E(T,r))\leq c_1 (R+r)^{\ell}(e^{c_2(4T-r^2)}+e^{-c_3T^{c_4}}),\qquad r,d>c_5,
\end{equation}
where all constants $c_1 ,\ldots, c_5$ are positive and independent of the parameters $d,r$ and $T$. Since the proof of this estimate follows exactly the same route as the proof of \cite[Lemma 4.4]{GKT21} (estimate for $T$) and \cite[Lemma 4.5]{GKT21} (estimate for $r$), we decided to omit the technical details. In particular, \eqref{eq:151122a} shows that for any $\varepsilon>0$ there is a choice of $T_0:=T(\varepsilon)$ and $r_0:=r(\varepsilon)$ such that $\PP(E(T_0,r_0))\ge 1- \varepsilon$. The same holds for $\Psi(\xi_{\ell,\infty})$. Further, we note that if a paraboloid $\Pi_{+,x}$ with $x\in\ext(\Psi(\xi_{\ell,d}))$ is such that $\Pi_{+,x}\cap \bd\Psi(\xi_{\ell,d})\subset B_{R+r_0}\times (-\infty, T_0]$, then
$$
x\in \{(v,h)\in\RR^{\ell+1}\colon h\leq T_0, \|v\|\leq R+r_0+\sqrt{T_0-h}\}=:K(R,\varepsilon).
$$
To complete the proof, it suffices to argue that there exists a coupling of $\xi_{\ell,d}$ and $\xi_{\ell,\infty}$ on a common probability space such that the probability that the restrictions of these processes to the region $K(R,\eps)$ do not coincide converges to $0$, as $d\to\infty$. In a suitable coupling, this probability is bounded above by a constant multiple of the $L^1$-norm of the difference of their intensity measures restricted to $K(R,\varepsilon)$, see~\cite[Theorem~3.2.2]{reiss_book}. As we have shown above, the densities $f_d(v,h)$ converge to $e^{\kappa^2\pi e h}$ as $d\to\infty$ uniformly on compact sets and hence pointwise. Also, by the inequality $1+x\leq e^x$, we have $f_d(v,h)\leq e^{c_6 h}$ for some absolute constant $c_6>0$. The fact that this upper bound is integrable over $K(r,\eps)$ has been shown in~\cite[Equation~(4.18)]{GKT21}. Thus, $f_d(v,h) \to e^{\kappa^2\pi e h}$ with respect to the $L^1$-norm on $K(R,\eps)$, which ensures that the required coupling of the Poisson processes indeed exists.  This completes the argument.
\end{proof}

Our aim is now to prove the weak convergence of the typical cell of the sectional Poisson-Voronoi tessellation to the typical cell of the Gaussian-Voronoi tessellation. Fix some $\ell\in \NN$ and let $\cC$ be the space of compact subsets of $\RR^\ell$ endowed with the Hausdorff metric. Put $\cC' = \cC\backslash\{\varnothing\}$.  The typical cells considered below are defined with respect to some fixed centre function $z: \cC' \to \RR^\ell$ in the sense of Section~\ref{sec:preliminaries_tess}, additionally satisfying $z(C) \in C$ for every $C\in \cC'$.

\begin{theorem}\label{theo:conv_typ_cell_corollary}
	Take any positive sequence $(\rho_d)_{d\in\NN}$ with $\lim_{d\to\infty}(\rho_d)^{1/d}=\kappa>0$.  Then, as $d\to\infty$, the distribution of the typical cell of the sectional Poisson-Voronoi tessellation  $\cW_{d,\rho_d}\cap \RR^\ell$ converges to the distribution of the typical cell of the Gaussian-Voronoi tessellation $\cG_{\ell, \kappa^2\pi e}$ weakly on $\cC$.
\end{theorem}

This theorem is a consequence of Theorem~\ref{theo:conv_tess} and the following general result.

\begin{proposition}\label{prop:conv_typ_cell_general}
	Let $(\cT_n)_{n\in \NN}$ be a sequence of stationary random tessellations on $\RR^\ell$ converging to a stationary random tessellation $\cT_\infty$ on $\RR^\ell$ in the following sense: All random tessellations are defined on a common probability space and
	for every $R>0$ we have $\lim\limits_{n\to\infty} \PP[A_n(R)] = 1$, where $A_n(R)$ is the event that the restrictions of $\cT_n$ and $\cT_\infty$ to $[-R,R]^\ell$ coincide. More precisely, $A_n(R)=A_n'(R)\cap A_n''(R)$ with
	\begin{align*}
		A_n'(R)
		&:=
		\{
		\text{for every }  C\in \cT_n \text{ s.t. }  C\cap [-R, R]^\ell\neq \varnothing  \text{ there is } C'\in \cT_\infty \text{ s.t. } C\cap [-R, R]^\ell = C'\cap [-R, R]^\ell
		\},\\
		A_n''(R)
		&:=
		\{
		\text{for every }  C'\in \cT_\infty \text{ s.t. }  C'\cap [-R, R]^\ell\neq \varnothing  \text{ there is } C\in \cT_n \text{ s.t. } C\cap [-R, R]^\ell = C'\cap [-R, R]^\ell
		\}.
	\end{align*}
	Also, suppose that the cell intensity of $\cT_n$ converges to that of $\cT_\infty$, that is $\lim\limits_{n\to\infty} \gamma_\ell(\cT_n) = \gamma_\ell(\cT_\infty)$, and that all these intensities are finite. Then, the distribution of the typical cell of $\cT_n$ converges to the distribution of the typical cell of $\cT_\infty$ weakly on $\cC$:
	\begin{equation}\label{eq:typ_cell_conv_weak}
		\PP^{z}_{\cT_n,\ell} {\overset{}{\underset{n\to\infty}\longrightarrow}} \PP^{z}_{\cT_\infty,\ell},
		\qquad
		\text{ weakly on } \cC.
	\end{equation}
\end{proposition}

\begin{proof}
	Let $f:\cC\to \RR$ be a bounded continuous function and recall that $z(C)$ is the centre of a cell $C$. Consider the random variables
	$$
	\xi_n := \sum_{C\in \cT_n, \, z(C)\in [0,1]^\ell} f(C),
	\qquad
	n\in \NN\cup\{\infty\}.
	$$
	By definition, the distribution of the typical cell of $\cT_n$ satisfies
	$$
	\int_\cC f \,\dd \PP^{z}_{\cT_n,\ell} = \frac {\EE \xi_n} {\gamma_\ell(\cT_n)},
	\qquad
	n\in \NN\cup\{\infty\}.
	$$
	Since $\lim\limits_{n\to\infty} \gamma_\ell(\cT_n) = \gamma_\ell (\cT_\infty)$, to prove~\eqref{eq:typ_cell_conv_weak} it suffices to check that
	\begin{equation}\label{eq:conv_typ_cells_proof_need}
		\lim_{n\to\infty} \EE \xi_n
		=
		\EE \xi_\infty.
	\end{equation}
	In the following we shall define a ``good'' event $C_n(R)$ on which $\xi_n$ and $\xi_\infty$ are equal. For $n\in \NN \cup\{\infty\}$ and $R>3$ consider the random event
	$$
	B_n(R) := \{ \nexists C\in \cT_n : C\cap [0,1]^\ell \neq \varnothing, C\not \subset[-R/2, R/2]^\ell\}.
	$$
	A cell $C$ with the properties listed in this definition is called a ``long cell'' in the tessellation $\cT_n$. The event $B_n(R)$ occurs if there is no long cell. On the event $A_n(R)$, each long cell $C\in \cT_n$ corresponds to a long cell $C'\in \cT_\infty$ with the same restriction to $[-R,R]^\ell$.
	It follows that $A_n(R) \cap (B_n(R))^c  \subset (B_\infty(R))^c$. Since all cells in $\cT_\infty$ are bounded almost surely, the maximal diameter of a cell in $\cT_\infty$ intersecting $[0,1]^\ell$ is some almost surely finite random variable $M$. It follows that
	$$
	\lim_{R\to\infty} \sup_{n\in\NN} \PP[A_n(R)\cap (B_n(R))^c] \leq \lim_{R\to\infty} \PP[(B_\infty(R))^c] \leq  \lim_{R\to\infty} \PP[M > (R/2) - 1] = 0.
	$$
	Recall also that $\lim_{n\to\infty} \PP[A_n(R)] = 1$ for every fixed $R>0$. Consider now the ``good'' event $C_n(R) := A_n(R) \cap B_n(R)$ and the ``bad'' event $D_n(R):= (C_n(R))^c$. It follows from the above that
	\begin{equation}\label{eq:lim_lim_D_n_R}
		\lim_{R\to\infty} \limsup_{n\to\infty} \PP[D_n(R)] = 0.
	\end{equation}
	Note that, on the event $C_n(R)$, the sets $\{C\in \cT_n: z(C)\in [0,1]^\ell\}$ and $\{C'\in \cT_\infty: z(C')\in [0,1]^\ell\}$ are equal, implying that $\xi_n = \xi_\infty$. Indeed, every cell $C\in \cT_n$ with $z(C)\in [0,1]^\ell$ is contained in $[-R/2, R/2]^\ell$ (since $z(C)\in C$ and there is no long cell) and, consequently, $C$ is also a cell of $\cT_\infty$ (since $A_n(R)$ occurs). Conversely, for every cell $C'\in \cT_\infty$ with $z(C') \in [0,1]^\ell$ there is a cell $C\in \cT_n$ such that the restrictions of $C$ and $C'$ to $[-R,R]^\ell$ coincide (since $A_n(R)$ occurs) and $C$ is not long (since $B_n(R)$ occurs), which implies that $C=C'$.
	It follows that
	$$
	\EE [{\bf 1}(C_n(R))\, \xi_n]
	=
	\EE [{\bf 1}(C_n(R))\, \xi_\infty],
	\qquad n\in \NN,
	\,
	R>3.
	$$
	To complete the proof of~\eqref{eq:conv_typ_cells_proof_need}, it suffices to verify that
	$$
	\lim_{R\to\infty} \limsup_{n\to\infty}
	\EE [{\bf 1}(D_n(R))\, \xi_n]
	=0,
	\qquad
	\lim_{R\to\infty} \limsup_{n\to\infty}
	\EE [{\bf 1}(D_n(R))\, \xi_\infty]
	=0.
	$$
	In view of~\eqref{eq:lim_lim_D_n_R}, it suffices to check that the family $\{\xi_n: n\in \NN\cup\{\infty\}\}$ is uniformly integrable. Now,
	$$
	|\xi_n|\leq \|f\|_\infty\eta_n
	\quad
	\text{ with }
	\quad
	\eta_n:= \sum_{C\in \cT_n, \, z(C)\in [0,1]^\ell} 1,
	\qquad
	n\in \NN\cup\{\infty\}.
	$$
As we already observed above, outside the event $D_n(R)$ we have $\eta_n = \eta_\infty$. From~\eqref{eq:lim_lim_D_n_R} it follows that $\eta_n\to\eta_\infty$ almost surely. On the other hand, we have $\EE \eta_n = \gamma_\ell(\cT_n) \to \gamma_\ell (\cT_\infty) = \EE \eta_\infty$ by assumption of the proposition (and all these expectations are finite). These two properties together with $\eta_n\geq 0$ imply  that the family $\{\eta_n: n\in \NN\cup\{\infty\}\}$ is uniformly integrable. Indeed, if this were not the case, we could find $\eps>0$ such that, after passing to a subsequence, $\EE [\eta_n {\bf 1} (Q_n)]>\eps$ for some events $Q_n$ with $\PP[Q_n] < 1/2^n$. Let $W_k:= \cup_{n=k}^\infty Q_n$. If $K$ is sufficiently large, then $\EE [\eta_\infty {\bf 1}(W_K)]<\eps/2$, while $\EE [\eta_n {\bf 1}(W_K)]>\eps$ for all $n\geq K$. Applying Fatou's lemma to the variables $(\eta_n {\bf 1}((W_K)^c))_{n\geq K}$ leads to a contradiction with the assumption $\EE \eta_n \to \EE \eta_\infty$. Finally, the bound $|\xi_n|\leq \|f\|_\infty\eta_n$ implies that  the family $\{\xi_n: n\in \NN\cup\{\infty\}\}$ is uniformly integrable as well, and the proof is complete.
\end{proof}

\subsection*{Acknowledgement}
ZK and CT were supported by the DFG priority program SPP 2265 \textit{Random Geometric Systems}. AG and ZK were supported by the DFG under Germany's Excellence Strategy  EXC 2044 -- 390685587, \textit{Mathematics M\"unster: Dynamics - Geometry - Structure}.

{\small\bibliographystyle{plain}
\bibliography{Bibliography}}

\end{document}